\documentclass[final]{article}     
\usepackage{amsmath,amssymb,amsfonts}
\usepackage{pstricks}
\usepackage{amssymb,amstext,latexsym, amsmath, amscd, amsfonts, cite}
\usepackage{graphicx}               
\usepackage{color}                  
\usepackage[T1]{fontenc}  
\usepackage{hyperref}
\usepackage{amsthm}
\usepackage{array}
\usepackage{amsfonts}
\providecommand{\abs}[1]{\left\lvert#1\right\rvert}
\providecommand{\norm}[1]{\left\Vert#1\right\Vert}
\newcommand{\sumfinie}{\sum_{i=N+1}^{\mathcal N}}

\newcommand{\pscal}[1]{\langle #1 \rangle}
\newcommand{\im}{\text{Im }}

\newcommand{\Var}{\mathbf{Var}}
\newcommand{\E}{\mathbf{E}}
\newcommand{\Emu}{\mathbf{E}_\mu}
\newcommand{\Dmu}{(\mu)}
\newcommand{\Ddmu}{(\mu,\Phi)}
\newcommand{\Dddmu}{(\mu,N,\Phi)}
\def \R{\mathbb{R}}

\def \N{\mathbb{N}}

\def \E{\mathbb{E}}

\def \Var{\hbox{{\rm Var}}}

\newtheorem{theorem}{Theorem}[section]

\newtheorem{cor}[theorem]{Corollary}

\newcommand*\samethanks[1][\value{footnote}]{\footnotemark[#1]}
\newcommand\remove[1]{}
\newcommand\add[1]{#1}
\newcommand\rcancel[1]{#1}
\newenvironment{keywords}{}{}
\newenvironment{AMS}{}{}
\title{Goal-oriented error estimation for \add{the} reduced basis method, with application to \remove{certified} sensitivity analysis}
\author{Alexandre Janon\thanks{Laboratoire de Math\'ematiques d'Orsay, Universit\'e Paris-Sud (France)} \and Ma\"elle Nodet\thanks{Laboratoire Jean Kuntzmann, Universit\'e Joseph Fourier, INRIA/MOISE} \and Cl\'ementine Prieur\samethanks[2]}

\begin{document}

\maketitle
\begin{abstract}
The reduced basis method is a powerful model reduction technique designed to speed up the computation of multiple numerical solutions of parametrized partial differential equations. We consider a quantity of interest, which is a linear functional of the PDE solution.  A new probabilistic error bound for the reduced model is proposed. It is efficiently and explicitly computable, and we show on different examples that this error bound is sharper than existing ones. We include application of our work to sensitivity analysis studies. 
\end{abstract}



\pagestyle{myheadings}
\thispagestyle{plain}
\markboth{JANON A., NODET M., PRIEUR C.}{GOAL-ORIENTED ERROR ESTIMATION}

\section*{Introduction}

A large number of mathematical models are based on partial differential equations (PDEs). These models require input data (e.g., the physical features of the considered system, the geometry of the domain, the external forces...) which enter in the PDE as \emph{parameters}. In many applications (for instance, design optimization, data assimilation, or uncertainty quantification), one has to numerically compute the solution of a parametrized partial differential equation for a large number of values of the parameters. In such a case, it is generally interesting, in terms of computation time, to perform all possible parameter-independent computations in an \emph{offline} phase, which is done only once, and to call an \emph{online} phase for each required value of the parameter, during which the information gathered in the offline phase can be used to speed up the computation of an approximate solution of the PDE, and, hence, to reduce the marginal (ie., per parameter) computation cost.

The reduced basis method \cite{nguyen2005certified} is a way of specifying such offline and online phases, which has been \remove{successully} \add{successfully} applied to various well-known PDEs \cite{grepl2005posteriori,knezevic2010reduced,veroy2005certified,janon2011certified}. One should note that, in the reduced basis (RB) method, the online phase does not compute a solution which is strictly identical to the numerical PDE solution, but an approximation of it, obtained by projecting the original discretized equations onto a well-chosen basis. In the application cases given above, however, one is not interested in the solution by itself, but rather in a \emph{quantity of interest}, or model \emph{output}, which is a functional of this solution. Taking this functional into account when performing the model reduction leads to a so-called \emph{goal-oriented} method. For instance, goal-oriented basis choice procedures have been tried with success in the context of dynamical systems in \cite{willcox2002balanced,ilak2008modeling}, where the basis is chosen so as to contain the modes that are relevant to accurately represent the output of interest, and in a general context in \cite{bui2007goal}, where the basis is chosen so as to minimize the overall output error. All those papers showed that using an adapted basis could lead to a great improvement of reduction error. 

This paper is about goal-oriented error estimation, that is, the description of a rigorous and computable  \emph{error bound} between the model output and the reduced one. Two different reduced model outputs can be considered: the first (which we call \emph{uncorrected} reduced output) is simply the output functional evaluated at the reduced output. The second (called \emph{corrected} reduced output), described in \cite{machiels1998general,Maday2002533,nguyen2005certified}, is the same, up to a correction term obtained from the solution of an auxiliary (dual) problem. The rate of convergence of the corrected output is better than the uncorrected one but the computation of the correction involves the application of the RB method to the dual problem, and this has generally the drawback of doubling offline \emph{and online} computational times. Regarding output error estimation, an error bound for the difference between the corrected reduced output and the original output is provided in the papers cited above. In this paper, we propose two new goal-oriented error bounds: one for the uncorrected reduced output, and one for the corrected reduced output. We also show, in numerical examples, that our bound is more precise than the existing bound.

This paper is organized as follows: in the first part, we describe our output error bounds and explain how to compute them; in the second part, we see how to apply our error bound to \remove{certified} sensitivity analysis studies; finally, the third and fourth parts present numerical applications.

\section{Methodology}
\subsection{Preliminaries}
\paragraph{Reference problem}
We begin by setting up the context of the reduced basis method for affine-parametrized linear partial differential equations presented in \cite{nguyen2005certified}. Our reference problem is the following: given a parameter tuple $\mu \in \mathcal P \subset \R^p$, and the vector space $X=\R^\mathcal N$ (for $\mathcal N\in\N$), find $u(\mu)$, the solution of:
\begin{equation}\label{e:refprob} A(\mu) u(\mu) =  f(\mu), \end{equation}
where $A(\mu)$ is an invertible square matrix of dimension $\mathcal N$, and $f(\mu) \in X$, then compute the \emph{output}:
\begin{equation}\label{e:out} s(\mu)=S(u(\mu)) \end{equation}
where $S: X \rightarrow \R$ is a linear form on $X$.
\medskip

\emph{Choice of the inner product:\hspace*{.25cm}}We suppose that $X=\R^\mathcal N$ (with the standard basis), is endowed with the standard Euclidean inner product: $\pscal{u,v}=u^t v$, with associated norm $\norm{u}=\sqrt{\pscal{u,u}}$
\medskip

\emph{Affine decomposition hypothesis:\hspace*{.25cm}}We suppose that $A(\mu)$ and $f(\mu)$ admit the following so-called affine decomposition \cite{nguyen2005certified}:
\begin{equation}\label{e:affdec} \forall \mu\in\mathcal P, \;\; A(\mu)=\sum_{q=1}^Q\Theta_q(\mu)A_q, \;\;
f(\mu)=\sum_{q'=1}^{Q'} \gamma_{q'}(\mu) f_{q'} \end{equation}
where $Q, Q'\in\N^*$,  $\Theta_q:\mathcal P\rightarrow\R$ and $\gamma_{q'}:\mathcal P \rightarrow \R$ (for $q=1,\ldots,Q$, $q'=1,\ldots,Q'$) are \remove{smooth} \add{given} functions, $A_q$ are square matrices of dimension $\dim X$ and $f_{q'}\in X$.

This hypothesis is required by the reduced basis method.

\paragraph{Reduced basis method}

The dimension of the finite element subspace $\dim X$ is generally fairly large, so that the numerical computation of $u(\mu)$ from the inversion of $A(\mu)$ is expensive. The reduced basis aims at speeding up ``many queries'', that is, the computation of $u(\mu)$ for all parameters $\mu\in\mathcal P_0$ where $\mathcal P_0$ is a finite but ``large'' subset of the parameter set $\mathcal P$. 
\medskip

\emph{Reduced problem: \hspace*{.25cm}} We consider a subspace $\widetilde X$ of $X$, and a matrix $Z$ whose columns are the components of a basis of $\widetilde X$ in a basis of $X$. This basis of $\widetilde X$ is called the \emph{reduced basis} in the sequel. We denote by $\widetilde u(\mu)$ the components, in the reduced basis, of the solution of the projection of \eqref{e:refprob} onto $\widetilde X$, that is, the solution of:
\begin{equation}\label{e:redprob} Z^t A Z \widetilde u(\mu) = Z^t f(\mu) \end{equation}
(where, for any matrix $M$, $M^t$ is the transpose of $M$).
\medskip

\emph{Choice of the reduced subspace:\hspace*{.25cm}} There are different techniques for choosing the reduced basis (the $Z$ matrix). This paper does not focus on this topic, but we cite the POD method (Proper orthogonal decomposition) \cite{sirovich1987turbulence}, and the Greedy method \cite{nguyen2005certified}. 
\medskip

\emph{Offline-online decomposition: \hspace*{.25cm}} The many-query computation can then be split into two parts: the first part (usually called the ``offline phase''), which is done only once, begins by finding a reduced subspace, then the $Q$ parameter-independent matrices:
\[ \widetilde A_q = Z^t A_q Z, \;\;q=1,\ldots,Q \]
and the $Q'$ vectors:
\[ \widetilde f_{q'}=Z^t f_{q'}, \;\; q'=1,\ldots,Q' \]
are computed and stored. In the second part (the ``online phase''), we compute, for each value of the parameter $\mu$:
\begin{equation}\label{e:rb1}  \widetilde A(\mu) = \sum_{q=1}^Q \Theta_q(\mu) \widetilde A_q, \;\;
\widetilde f(\mu) = \sum_{q'=1}^{Q'} \gamma_q(\mu) \widetilde f_{q'} \end{equation}
and solve for $\widetilde u(\mu)$ satisfying:
\begin{equation}\label{e:rb2}  \widetilde A(\mu)\widetilde u(\mu)=\widetilde f(\mu). \end{equation}
The key point is that the operations in \eqref{e:rb1} and \eqref{e:rb2} are performed on vectors and matrices of size $\dim \widetilde X$, and that the complexity of these operations is totally independent from the dimension of the underlying ``truth'' subspace $X$. In many cases, the smoothness of the map $\mu\mapsto u(\mu)$ allows to find (in a constructive way, ie., compute) $\widetilde X$ so that $\dim\widetilde X \ll \dim X$ while keeping $\norm{u(\mu)-Z \widetilde u(\mu)}$ small, hence enabling significant computational savings. 
\medskip

\emph{Output approximation:\hspace*{.25cm}} The output $s(\mu)$ can also be approximated from $\widetilde u(\mu)$ using an efficient offline-online procedure: let $l\in X$ be so that:
\[ s(u(\mu))=\pscal{l,u} \;\;\forall u\in X; \]
in the offline phase we compute and store:
\[ \widetilde l=Z^t l \]
and in the online phase we take:
\[ \widetilde s(\mu)=\pscal{\widetilde l, \widetilde u(\mu)} \]
as an approximation for $s(\mu)$.

\paragraph{Reduced-basis error bounds\medskip\\*}
\hspace*{-.35cm}\emph{Bound on $u$:\hspace*{.25cm}}
Under additional coercivity hypothesis on $A$, the reduced basis method \cite{nguyen2005certified} also provides an efficient offline-online procedure for computing $\epsilon^u(\mu)$ so that the approximation can be \emph{certified}. This bound is based on the \emph{dual norm of the residual}:
\[ \rho(\mu) = \add{\norm{r(\mu)}_{RB} = }\norm{ A(\mu) Z \widetilde u(\mu) - f(\mu) }_{RB} \]
where $\norm{\cdot}_{RB}$ is a suitably chosen norm on $X$ (not necessarily $\norm{\cdot}$), and a \emph{stability constant} bound, which can be written as:
\begin{equation}\label{e:stabconst} 0 < \alpha(\mu) \leq \inf_{v \in X, \norm{v}_{RB}=1} \abs{v^t A(\mu) v} \end{equation}
when $A$ is symmetric. The inequality sign is due to the fact that the exact infimum can be costly to evaluate in the online stage; usually a procedure such as the successive constraints method \cite{huynh2007successive} is used in order to quickly find a lower bound.

The bound $\epsilon^u(\mu)$ reads:
\[ \forall \mu\in\mathcal P\;\; \norm{u(\mu)-Z \widetilde u(\mu)}_{RB} \leq \frac{\rho(\mu)}{\alpha(\mu)} := \epsilon^u(\mu), \]
The online procedure for the computation of $\epsilon(\mu)$ is also of complexity independent of $\dim X$.
\medskip

\emph{Lipschitz bound on $s$:\hspace*{.25cm}}
This online error bound can in turn be used to provide a certification on the output:
\begin{equation}\label{e:bornaive}\forall \mu\in\mathcal P\;\;  \abs{s(\mu)-\widetilde s(\mu)} \leq \underbrace{\norm{l}_{RB} \epsilon^u(\mu)}_{=:\epsilon^L(\mu)}\end{equation}
We call this bound the ``Lipschitz'' bound, and denote it by $\epsilon^L(\mu)$. \add{It is well-known that this bound is very pessimistic.}

\bigskip

The aim of Section \ref{ss:errdecomp} is to bound $\abs{s(\mu)-\widetilde s(\mu)}$ by a quantity which is smaller than $\epsilon^L(\mu)$ of \eqref{e:bornaive} and can be computed using an efficient offline-online procedure which does not require computation of $\epsilon^u(\mu)$, described in Section \ref{ss:montecarloapprox}. In Section \ref{ss:boundcorrected}, we consider a better approximation of $s(\mu)$ (denoted by $\widetilde s_c(\mu)$) which also depends on the solution of the adjoint equation of \eqref{e:refprob} projected on a suitably selected dual reduced basis, and we see how the proposed bound for $\abs{s(\mu)-\widetilde s(\mu)}$ can be modified in order to bound $\abs{s(\mu)-\widetilde s_c(\mu)}$. 

\subsection{Probabilistic error bound}
\label{ss:errdecomp}
In this section, we give the expression of our output error bound. We \remove{begin with some notation: {let's} denote the residual by} \add{recall the notation for the residual} $r(\mu)$:
\[ r(\mu) = A(\mu) Z \widetilde u(\mu) - f(\mu) \in X,  \]
and the adjoint problem solution (which will naturally appear in the proof of Theorem \ref{t:1}) by $w(\mu)$:
\[ w(\mu) = A(\mu)^{-t} l. \]
Let, for any orthonormal basis $\Phi = \{ \phi_1,\ldots, \phi_{\mathcal N} \}$ of $X$, any $N\in\N^*$, and $i=1,\ldots,N$,
\[ D_i(\mu,\Phi)=\pscal{w(\mu), \phi_i}. \]
We take a partition $\{\mathcal P_1, \ldots, \mathcal P_K\}$ of the parameter space $\mathcal P$, that is:
\[ \mathcal P = \cup_{k=1}^K \mathcal P_k \;\;\;\text{and}\;\;\; k\neq k' \remove{\rightarrow} \add{\Rightarrow}\mathcal P_k \cap \mathcal P_{k'}=\emptyset. \]
We set, for $i=1,\ldots,N$ and $k=1,\ldots,K$:
\[ \beta_{i,k}^{min}(\Phi) = \min_{\mu\in\mathcal P_k} D_i(\mu\add{,\Phi}), \;\;\; \beta_{i,k}^{max}(\Phi) = \max_{\mu\in\mathcal P_k} D_i(\mu\add{,\Phi}), \]
and:
\[ \beta_i^{up}\Ddmu = \left\{ \begin{array}{l} \beta_{i,k(\mu)}^{max}(\Phi) \text{ if } \pscal{r\Dmu,\phi_i}>0 \\ \beta_{i,k(\mu)}^{min}(\Phi) \text{ else, } \end{array} \right. \]
\[ \beta_i^{low}\Ddmu = \left\{ \begin{array}{l} \beta_{i,k(\mu)}^{min}(\Phi) \text{ if } \pscal{r\Dmu,\phi_i}>0 \\ \beta_{i,k(\mu)}^{max}(\Phi) \text{ else, } \end{array} \right. \]
where $k(\mu)$ is the only $k$ in $\{1,\ldots,K\}$ so that $\mu\in\mathcal P_k$.
We also set:
\[ T_1^{low}\Dddmu = \sum_{i=1}^N \pscal{r\Dmu, \phi_i} \beta_i^{low}\Ddmu, \;\;
   T_1^{up}\Dddmu = \sum_{i=1}^N \pscal{r\Dmu, \phi_i} \beta_i^{up}\Ddmu, \]
\[ T_1\Dddmu=\max\left( \abs{T_1^{low}\Dddmu}, \abs{T_1^{up}\Dddmu} \right). \]
Finally, we suppose that $\mu$ is a random variable on $\mathcal P$ and set:
\[ T_2(N,\Phi)=\Emu\left(\abs{ \sumfinie \pscal{w(\mu), \phi_i} \pscal{r(\mu),\phi_i} } \right). \]

We have the following theorem:
\begin{theorem} \label{t:1}
	For any $\alpha\in]0;1[$ and for any $N\in\N^*$, we have:
	\[ P\left(\abs{s(\mu)-\widetilde s(\mu)} > T_1\Dddmu + \frac{T_2(N,\Phi)}{\alpha} \right) \leq \alpha. \]
\end{theorem}

{\em Proof:}
We begin by noticing that:
\[ A(\mu)^{-1} r(\mu) = Z \widetilde u(\mu) - u(\mu) \]
so that:
\[ \widetilde s(\mu) - s(\mu) = \pscal{l, Z \widetilde u(\mu) - u(\mu)} = \pscal{l, A(\mu)^{-1} r(\mu)} = \pscal{w(\mu), r(\mu)}. \]
We expand the residual in the $\Phi$ basis:
\[ r(\mu) = \sum_{i\geq1} \pscal{r(\mu), \phi_i} \phi_i. \]
Hence:
\begin{equation}
\label{e:decerr}
\widetilde s(\mu) - s(\mu) \rcancel{= \sum_{i\geq1} \pscal{l, A(\mu)^{-1} \phi_i} \pscal{r(\mu), \phi_i}} = \sum_{i\geq1} \pscal{w(\mu), \phi_i} \pscal{r(\mu),\phi_i}.
\end{equation}
We clearly have that for any $N\in\N^*$:
\[ 
	\sum_{i=1}^N \pscal{r\Dmu, \phi_i} \beta_i^{low}\Ddmu 
	\leq
 \sum_{i=1}^N \pscal{r\Dmu, \phi_i} \pscal{w(\mu),\phi_i}  \leq 
	 \sum_{i=1}^N \pscal{r\Dmu, \phi_i} \beta_i^{up}\Ddmu
	 \]
and this implies:
\begin{equation}
\label{e:myeqn}
\abs{  \sum_{i=1}^N \pscal{r\Dmu,\phi_i} \pscal{w(\mu), \phi_i} } \leq T_1\Dddmu. 
\end{equation}
So we have:       
\begin{align*}
P&\left(\abs{s(\mu)-\widetilde s(\mu)} > T_1\Dddmu + \frac{T_2(N,\Phi)}{\alpha} \right) \\
\leq P&\left( \abs{s(\mu)-\widetilde s(\mu)} > \abs{\sum_{i=1}^N \pscal{r\Dmu,\phi_i} \pscal{w\Dmu, \phi_i}  } + \frac{T_2(N,\Phi)}{\alpha} \right) \text{ by \eqref{e:myeqn}} \\
= 	P&\left( \abs{s(\mu)-\widetilde s(\mu)} -\abs{\sum_{i=1}^N \pscal{r\Dmu,\phi_i} \pscal{w\Dmu, \phi_i}  } > \frac{T_2(N,\Phi)}{\alpha} \right) \\
\leq P&\left( \abs{\sumfinie  \pscal{r\Dmu,\phi_i} \pscal{w\Dmu, \phi_i} } > \frac{T_2(N,\Phi)}{\alpha} \right) \text{ by \eqref{e:decerr}}\\
\leq  \alpha& \text{ thanks to Markov's inequality. } \qquad\qed 
\end{align*}

\paragraph{Choice of $\Phi$} The error bound given in Theorem \ref{t:1} above is valid for any orthonormal basis $\Phi$. For efficiency reasons, we would like to choose $\Phi$ so that the parameter-independent part $T_2(N,\Phi)$ is the smallest possible, for a fixed truncation index $N\in\N^*$.

To our knowledge, minimizing $T_2(N,\Phi)$ over orthonormal bases of $X$ is an optimization problem for which no efficient algorithm exists. However, we can minimize an upper bound of $T_2(N,\Phi)$.

We define an auto-adjoint, positive operator $G: X \rightarrow X$ by:
\begin{equation}
\label{e:defg}
 \forall \phi\in X,\;\; G\phi =  \frac{1}{2} \Emu\left( \pscal{ r(\mu), \phi } r(\mu) + \pscal{ w(\mu), \phi } w(\mu)  \right).
\end{equation}
Let $\lambda_1 \geq \lambda_2 \geq \ldots \lambda_{\mathcal N}\geq 0$ be the eigenvalues of $G$. Let, for $i\in\{1,2\ldots,\mathcal N\}$, $\phi_i^G$ be an unit eigenvector of $G$ associated with the $i^\textrm{th}$ eigenvalue, and $\Phi^G = \{ \phi_1^G, \ldots, \phi_{\mathcal N}^G \}$. 

We can state that:
\begin{theorem}
\[ T_2(N, \Phi^G) \leq \sumfinie \lambda_i^2. \]
\end{theorem}
\begin{proof}
We have:
\[
T_2(N,\Phi) \leq \frac{1}{2} \Emu\left( \sumfinie \pscal{w(\mu), \phi_i}^2 + \sumfinie \pscal{r(\mu),\phi_i}^2 \right) =: T_2^{sup}(N,\Phi) = \sumfinie \pscal{G \phi_i, \phi_i}
\]
Using Theorem 1.1 of \cite{volkwein1999proper}, we get that the minimum of $T_2^{sup}(N,\Phi)$ is attained for $\Phi=\Phi^G$, and that minimum is $\sumfinie \lambda_i^2$.
\end{proof}

This theorem suggests to use $\Phi=\Phi^G$, so as to control $T_2(N,\Phi)$.

\subsection{Monte-Carlo approximation of the error bound}
\label{ss:montecarloapprox}
In this Subsection, we present an implementable offline/online procedure for the estimation of the upper bound for $\abs{\widetilde s(\mu)-s(\mu)}$ presented in Theorem \ref{t:1}. 

\paragraph{Estimation of $\phi_i^G$}
We fix a truncation index $N\in\N^*$, and we estimate $\{\phi_i^G\}_{i=1,\ldots,N}$ by using a modification of the method of snapshots used in Proper Orthogonal Decomposition \cite{sirovich1987turbulence}. This estimation is performed during the offline phase. We begin by estimating the $G$ operator by $\widehat G$, then we approximate $\phi_i^G$ by the appropriate eigenvectors of $\widehat G$. 
\medskip

\emph{Estimation of $G$:} \hspace*{.25cm}
We take a finite (large), subset of parameters $\Xi \subset \mathcal P$, randomly sampled from the distribution of the parameter, and we approximate the $G$ operator by:
\[ \widehat G \phi = \frac{1}{2 \#\Xi} \sum_{\mu\in\Xi} \left( \pscal{r(\mu),\phi}r(\mu)+\pscal{w(\mu),\phi}w(\mu) \right) \]
In other words, $\widehat G$ is a Monte-Carlo estimator of $G$. We take $\{\widehat \phi_i^G\}_{i=1,\ldots,N}$ as the unit eigenvectors associated with the $N$ largest eigenvalues of $\widehat G$.
\medskip

\emph{Computation of the eigenvalues of $\widehat G$:} \hspace*{.25cm} The operator $\widehat G$ admits the following matrix representation:
\[ \widehat G=\frac{1}{2 \#\Xi} \left( WW^t+RR^t \right), \]
where $W$ (resp. $R$) is the matrix whose columns are the components of $w(\mu)$ (resp. $r(\mu)$) in a basis of $X$, for $\mu\in\Xi$. These two matrices have $\#\Xi$ columns and $\dim X$ lines, which means that the matrix above is $\dim X \times \dim X$. 

In general, we take $\#\Xi\ll\dim X$, and so it is computationally advantageous to notice that if $\phi$ is an eigenvector of $\widehat G$ associated with a nonzero eigenvalue $\lambda$, then:
\[ \frac{1}{\lambda}\frac{1}{2\#\Xi} \left(  (WW^t\phi+RR^t\phi) \right)=\phi, \]
so that $\phi\in\im W +\im R=:\mathcal V$. Hence, if $V$ is the matrix of an orthonormal basis of $\mathcal V$, then there exists $\psi$ so that $\phi=V\psi$ and we have:
\[ WW^t\phi+RR^t\phi=\lambda\phi \;\; \Longrightarrow \;\; \left[ V^t \frac{1}{2\#\Xi} \left( (WW^t+RR^t) \right) V \right] \psi=\lambda \psi. \]
As a consequence, it is sufficient to find the dominant eigenvectors $\widehat\psi_1^G,\ldots,\widehat\psi_N^G$ of the matrix $\Sigma=\frac{1}{2\#\Xi} V^t(WW^t+RR^t)V $ (of size $2 \add{\#}\Xi$), and to deduce $\widehat\phi_i^G$ from $\widehat\psi_i^G$ by the relation $\widehat\phi_i^G = V \widehat\psi_i^G$. Besides, by writing $\Sigma$ as:
\[\Sigma=\frac{1}{2\#\Xi} \left( (V^t W)(W^tV)+(V^tR)(R^tV) \right), \]
it is possible to compute and store $\Sigma$ without storing nor computing any dense $\dim X \times \dim X$ matrix.

\paragraph{Computation of $T_1\Dddmu$} All the quantities intervening in $T_1\Dddmu$ can be straightforwardly deduced from $\beta_{i,k}^{min,max}$ and $\pscal{r\Dmu,\widehat\phi_i^G}$. \medskip 

\emph{Computation of $\beta_{i,k}^{min}(\Phi)$ and $\beta_{i,k}^{max}(\Phi)$: } \hspace*{.25cm}
For $i=1,\ldots,N$ and $k=1,\ldots,K$, the \mbox{reals} $\beta_{i,k}^{min}(\Phi)$ and $\beta_{i,k}^{max}(\Phi)$ can be computed during the offline phase, as they are parameter-independent. \add{If the $\Theta_{q}$s functions are smooth enough, }\remove{Thanks to the availability of the gradient of $D_i(\mu)$ with respect to $\mu$,} a quasi-Newton optimization such as L-BFGS \cite{LbfgsB} can be used so as to compute these reals; they can also be approximated by a simple discrete minimization:
\begin{equation}
\label{betatilde}
\widetilde\beta_{i,k}^{min}(\Phi)=\min_{\mu\in\Xi\cap\mathcal P_k} D_i(\mu,\Phi), \;\;\;\; 
\widetilde\beta_{i,k}^{max}(\Phi)=\max_{\mu\in\Xi\cap\mathcal P_k} D_i(\mu,\Phi).
\end{equation}
One should note here that one has to be careful, as these numerical optimisation procedures are not exact. Indeed, the discrete optimisation is only an approximation, and the Quasi-Newton optimisation may be handled carefully, so as to avoid being trapped in local extrema.

\medskip

\emph{Computation of $\pscal{r\Dmu,\widehat\phi_i^G}$: } \hspace*{.25cm} We denote by  $\{\zeta_1,\ldots,\zeta_n\}$ are the column vectors of $Z$, which form a basis of the reduced space $\widetilde X$. We can, during the offline phase, compute the following parameter-independent quantities:
\[ \pscal{f_{q'},\widehat \phi_i^G}, \pscal{A_q \zeta_j, \widehat \phi_i^G} \; (i=1,\ldots,N,\, j=1,\ldots,n,\, q=1,\ldots,Q,\, q'=1,\ldots,Q'). \]
Let a parameter $\mu\in\mathcal P$ be given, and $\widetilde u_1\Dmu, \ldots, \widetilde u_n\Dmu$ be the components of the reduced solution $\widetilde u\Dmu$ in the reduced basis $\{\zeta_1,\ldots,\zeta_n\}$.

By using the relation:
\[ \pscal{r\Dmu,\widehat\phi_i^G} = \sum_{q=1}^Q \Theta_q(\mu) \sum_{j=1}^n \widetilde u_j\Dmu \pscal{A_q\zeta_j,\widehat\phi_i^G} - \sum_{q'=1}^{Q'} \gamma_{q'}(\mu) \pscal{f_{q'},\phi_i^G}, \]
the dot products between the residual and $\widehat\phi_i^G$ can be computed in the online phase, with a complexity of $O(nQ+Q')$ arithmetic operations, $O(Q)$ evaluations of $\Theta$ functions and $O(Q')$ evaluations of $\gamma$ functions, which is independent of $\dim X$.  \medskip

\paragraph{Estimation of $T_2(N,\Phi)$}
We approximate $T_2(N,\Phi)$ by computing the following Monte-Carlo estimator:
\[ \widehat T_2(N,\Phi) = \frac{1}{2\#\Xi} \sum_{\mu\in\Xi} \abs{ \widetilde s(\mu)-s(\mu)-\sum_{i=1}^N  \pscal{w(\mu), \phi_i} \pscal{r(\mu),\phi_i} }. \]
As this quantity is $\mu$-independent, it can be computed once and for all during the offline phase. 


\paragraph{Final error bound}
By using Theorem \ref{t:1}, we get that for $\epsilon(\mu,\alpha,N,\Phi)=T_1\Dddmu+T_2(N,\Phi)/\alpha$, we have:
\[ P \left( \abs{s(\mu)-\widetilde s(\mu)} \geq \epsilon(\mu,\alpha,N,\Phi) \right) \leq \alpha, \]
so we may take, as estimated (computable) error bound with risk $\alpha$,
\begin{equation}\label{e:defheps} \widehat\epsilon(\mu,\alpha,N,\Phi)=T_1\Dddmu+\frac{\widehat T_2(N,\Phi)}{\alpha}. \end{equation}
In the rest of the text, this computable error bound is designated as the \emph{error bound on the non-corrected output}, by contrast to the bound described in Section \ref{ss:errdecomp}.

One may note that estimating $T_2(N,\Phi)$ by $\widehat T_2(N,\Phi)$ causes some error on the risk of the computable error bound $\widehat\epsilon(\mu,\alpha,N,\Phi)$. This error is analyzed in Appendix \ref{appendix}.

\subsection{Bound on the corrected output}
\label{ss:boundcorrected}
The reduced output $\tilde s(\mu)$ is a natural reduced output that approximates $s(\mu)$. It is possible to solve an auxiliary problem in order to compute an error correction that improves the order of convergence of the reduced output. As we will see, the error bound presented above can be easily modified so as to certify the corrected output.

\paragraph{Output correction}
The idea of solving an adjoint problem in order to improve the order of convergence of a scalar output has originated in \cite{machiels1998general}, with the first application in the reduced basis context in \cite{Maday2002533}. We introduced the so-called adjoint problem, whose solution \remove{$u_d(\mu)$} \add{$w(\mu)$} satisfies:
\[ A(\mu)^t \rcancel{u_d}\add{w}(\mu) = l, \]
and the solution $ \rcancel{\tilde u_d} \add{\tilde w}(\mu)$ of the reduced adjoint problem:
\[ Z_d^t A(\mu)^t Z_d \rcancel{\tilde u_d} \add{\tilde w}(\mu) = Z_d^t l, \]
where $Z_d$ is the selected matrix of the reduced basis for the adjoint basis.

The corrected output is:
\[ \tilde s_c (\mu) = \tilde s(\mu) - \pscal{Z \rcancel{\tilde u_d} \add{\tilde w}(\mu), r(\mu)} \]

One shoud note that the corrected output can be computed using an efficient offline-online procedure which requires \emph{two} reduced basis solutions, hence roughly doubling (when $Z_d$ has the same number of columns than $Z$) the offline and online computation times, except in the particular case where $A$ is symmetric and $l$ is proportional to $f$.

\paragraph{Existing error bound on the corrected output}
In \cite{nguyen2005certified}, it is shown that:
\begin{equation}
\label{e:borneMPdual}
 \abs{s(\mu) - \tilde s_c(\mu)} \leq \frac{\norm{r(\mu)}' \norm{r_d(\mu)}'}{\alpha(\mu)} =: \epsilon_{cc}(\mu) 
\end{equation}
where $\alpha(\mu)$ is the stability constant bound defined at \eqref{e:stabconst} (in the symmetric case) and $r_d(\mu)$ is the dual residual:
\[ r_d(\mu) = A^t(\mu) Z_d \rcancel{\tilde u_d} \add{\tilde w}(\mu) - l(\mu). \]
Hereafter, the existing bound $\epsilon_{cc}$ is called \emph{dual-based error bound}.

\paragraph{Proposed error bound}
It is clear that the work performed in the above sections can be reused so as to provide a probabilistic bound on $\abs{s(\mu)-\tilde s_c(\mu)}$, by simply replacing $w(\mu)$ by:
\begin{equation}
\label{e:changeW}
 w_c(\mu) = w(\mu) - Z \rcancel{\tilde u_d} \add{\tilde w}(\mu), 
\end{equation}
and hence giving a competitor (called \emph{error bound on the corrected output}) to the dual-based error bound.

\subsection{Summary of the different bounds}
To sum up, we have presented four output computable error bounds. Two of them are bounds for the uncorrected ouput error:
\begin{itemize}
\item the ``Lipschitz'' error bound $\epsilon^L$ \eqref{e:bornaive};
\item the estimated error bound on the uncorrected output $\widehat\epsilon$, that we propose in \eqref{e:defheps};
\end{itemize}
and two of them for the corrected output error:
\begin{itemize}
\item the existing dual-based error bound $\epsilon_{cc}$, defined at \eqref{e:borneMPdual};
\item the estimated error bound on the corrected output $\widehat\epsilon_c$, that is $\eqref{e:defheps}$ amended with \eqref{e:changeW}.
\end{itemize}

\paragraph{Comparison of the various presented bounds.}

We can now discuss and compare the different considered error bounds (numerical comparisons are gathered in sections \ref{s:numres1} and \ref{s:numres2}).

\emph{Online costs:} using $\epsilon^{cc}$ or $\widehat\epsilon^c$ will require a double computation time (except if we are in a ``compliant'' case, i.e. $l=f$), when compared with using $\epsilon^L$ or $\widehat\epsilon$. However, this computational supplement enables error correction of the reduced output and an improved error estimation. Except for the possible supplemental cost for the adjoint problem resolution, the online cost of $\widehat\epsilon^c$ or $\widehat\epsilon$ should be comparable to the cost of $\epsilon^{cc}$, if not slightly better for the former, as no optimization problem has to be solved in the online phase (while $\epsilon^{cc}$, when using the Successive  Constraints Method (SCM) \cite{huynh2007successive}, requires a linear programming in order to compute an online $\alpha(\mu)$).

\emph{Offline costs:} the SCM procedure used for $\epsilon^{cc}$ and $\epsilon^L$ requires the offline resolution of a number of large eigenproblems on the ``finite element space'' $X$. Depending on the problem at hand, this may, or may not be more expensive than the optimization problems required to compute the $\widetilde\beta$ constants (see \eqref{betatilde}) and the Monte-Carlo estimation of $\widehat T_2$. One of the methods can also be feasible and the other not.

\emph{Probabilistic vs. deterministic:} one can also notice that the $\widehat\epsilon^c$ and $\widehat\epsilon$ bounds are probabilistic in nature; they  increase when the risk level decreases. In practice, as shown in the numerical experiments of Section 4, our probabilistic bound is much shaper, even while choosing a very small risk level.

\emph{Accuracy:} the classical bound introduces the dual norm of the residual, which causes a loss of accuracy. In this work, we avoid this step by using
the majoration in \eqref{e:myeqn} and a probabilistic argument.   \medskip

To conclude this discussion, one can say that there may not be a definitive winner error bound, and that the best choice highly depends on the problem at hand (dimension of the $X$ space, numbers of terms in the affine decompositions), the computational budget, the required precision, the number of online queries to perform, and the probability of failure of the error bound that one can afford.

\section{Application to sensitivity analysis}
\label{s:appSA}
Our error estimation method is applied in sensitivity analysis, so as to quantify the error caused by the replacement of the original model output by the reduced basis output during the Monte-Carlo estimation of the Sobol indices. For the sake of self-completeness, we briefly present the aim and the computation of these indices, and we refer to \cite{saltelli-sensitivity}, \cite{saltelli2002making} and \cite{janon:inria-00567977} for details.

\subsection{Definition of the Sobol indices}
For $i=1,\ldots,p$, the $i^\textrm{th}$ Sobol index of a function of $p$ variables $s(\mu_1,\ldots,\mu_p)$ is defined by:
\begin{equation}
\label{e:defsi}
 S_i = \frac{\Var\left( \E(s(\mu_1,\ldots,\mu_p)|\mu_i) \right)}{\Var\left( s(\mu_1,\ldots,\mu_p) \right)}, 
\end{equation}
the variances and conditional expectation being taken with respect to a postulated distribution of the $(\mu_1,\ldots,\mu_p)$ input vector accounting for the uncertainty on the inputs' value. These indices are well defined as soon as $s \in L^2(\mathcal P)$ and \remove{that} $\Var\left( s(\mu_1,\ldots,\mu_p) \right) \neq 0$. When $\mu_1, \ldots, \mu_p$ are (stochastically) independent, the $i^\textrm{th}$ Sobol index can be interpreted as the fraction of the variance of the output that is caused by the uncertainty on the $i^\textrm{th}$ parameter $\mu_i$. All the Sobol indices lie in $[0;1]$; the closer to zero (resp., one) $S_i$ is, the less (resp., the more) importance $\mu_i$'s uncertainty has on $s$'s uncertainty.

\subsection{Estimation of the Sobol indices}
The conditional expectation and variances appearing in \eqref{e:defsi} are generally not amenable to analytic computations. In those cases, one can estimate $S_i$ by using a Monte-Carlo estimate: from two random, independent samples of size $M$ of the inputs' distribution, we compute $2M$ appropriate evaluations $\{s_j\}$ and $\{s_j'\}$ of $s$, and estimate $S_i$ by:
\begin{equation}
\label{e:defesti}
\widehat S_i = \frac{ \frac{1}{M} \sum_{j=1}^{M} s_j s_j'   - \left(\frac{1}{M} \sum_{j=1}^{M} s_j \right) \left(\frac{1}{M}\sum_{j=1}^{M}  s_j' \right) }{ \frac{1}{M}\sum_{j=1}^{M}  s_j^2 - \left( \frac{1}{M} \sum_{j=1}^{M}  s_j \right)^2 }.
\end{equation}

When $M$ and/or the required time for the evaluation of the model output are large, it is computationally advantageous to replace $s$ by its surrogate model $\widetilde s$. By using \eqref{e:defesti} on $\widetilde s$ (hence with reduced model outputs $\{\widetilde s_j\}$ and $\{\widetilde s_j'\}$), one estimates the Sobol indices \emph{of the surrogate model} rather than those of the true model. We presented in  \cite{janon:inria-00567977}, Sections 3.1 and 3.2, a method to quantify the error made in the Sobol index estimation when replacing the original model by the surrogate one. We defined two estimators $\widehat S_{i,\alpha_{as}/2}^m$ and $\widehat S_{i,1-\alpha_{as}/2}^M$, relying on output error bound samples $\{\epsilon_j\}$ and $\{\epsilon_j'\}$, and proved that:
\begin{theorem}
If:
\[ \forall j=1,\ldots,M, \;\;\; \abs{ s_j - \widetilde s_j } \leq \epsilon_j \;\;\text{and}\;\; \abs{ s_j'-\widetilde s_j' }\leq\epsilon_j', \]
then we have:
\[ P\left( S_i \in [\widehat S_{i,\alpha_{as}/2}^m; \widehat S_{i,1-\alpha_{as}/2}^M ] \right) \geq 1- \alpha_{as}. \]
\end{theorem}

In our case, the output error bound $\epsilon(\mu)$ of Theorem \ref{t:1} does not satisfy the above hypothesis, but satisfies a weaker ``probabilistic'' statement. This is the object of the following Corollary:
\begin{cor}
\label{corolle}
If:
\[ \forall j=1,\ldots,M,\;\;\; P\left( \abs{s_j - \widetilde s_j} \geq \epsilon_j \right) \leq \alpha
\;\;\text{and}\;\;
\forall j=1,\ldots,M,\;\;\; P\left( \abs{s_j' - \widetilde s_j'} \geq \epsilon_j' \right) \leq \alpha,
 \]
then we have:
\[ P\left( S_i \in [\widehat S_{i,\alpha_{as}/2}^m; \widehat S_{i,1-\alpha_{as}/2}^M ] \right) \geq (1- \alpha_{as}) \times (1-\alpha)^{2M}. \]
\end{cor}
\begin{proof}
We easily have that:
\begin{eqnarray*}
 P\left( S_i \in [\widehat S_{i,\alpha_{as}/2}^m; \widehat S_{i,1-\alpha_{as}/2}^M ] \right) &\geq&
  P\left( S_i \in [\widehat S_{i,\alpha_{as}/2}^m; \widehat S_{i,1-\alpha_{as}/2}^M ] \,|\, \forall j, \abs{s_j-\widetilde s_j}<\epsilon(\mu) \right)
  \\ &&\times P \left( \forall j, \abs{s_j-\widetilde s_j}<\epsilon(\mu) \right) \\
  &\geq& (1-\alpha_{as})\times(1-\alpha)^{2M}.
\end{eqnarray*}
\end{proof}

\section{Numerical results I: Diffusion equation}
\label{s:numres1}
Important instances of problem \eqref{e:refprob} appear as discretizations of $\mu$-parametrized linear partial differential equations (PDE); the $X$ space is typically a finite element subspace (e.g., Lagrange $P^1$ finite elements), which we still see as identical to $\R^\mathcal N$. $A(\mu)$ and $f$ are given by Galerkin projection of the weak form of the PDE onto a suitable basis of $X$. The boundary conditions of the PDE are usually either encoded in $X$ or in $A(\mu)$, and the inner product used to perform the Galerkin projection is typically the $L^2$ or $H^1$ inner product. The use of the standard Euclidean product is justified by the fact that the relevant functional inner product has already been used to write $A(\mu)$ and $f(\mu)$, and the discrete matricial problem can be considered using the Euclidean inner product.

\subsection{Benchmark problem}
Our benchmark problem \cite{rozza2008venturi} is the following: given a parameter vector
\[ \mu=(\mu_1,\mu_2,\mu_3) \in \mathcal P= [0.25, 0.5] \times [2,4] \times [0.1,0.2], \]
we consider the domain $\Omega=\Omega(\mu)$ below:
\begin{center}
\begin{pspicture}(0,-1)(10,5)
\psline(0,0)(10,0)
\psline(0,5)(2.6,5)
\psarc(2.6,4.6){.4}{0}{90}
\psline(7.4,5)(10,5)
\psarc(7.4,4.6){.4}{90}{180}
\psline(3,4.6)(3,3.4)
\psarc(3.4,3.4){.4}{180}{270}
\psline(3.4,3)(6.6,3)
\psarc(6.6,3.4){.4}{270}{360}
\psline(7,3.4)(7,4.6)
\psline{<->}(0.4,0)(.4,5)
\uput[r](.4,2.5){$1$}
\psline{<->}(4.4,0)(4.4,3)
\uput[l](4.4,1.5){$\mu_1$}
\psline{<->}(3,3.5)(7,3.5)
\uput[u](5,3.5){$\mu_2$}
\psline{<->}(0,3.5)(3,3.5)
\uput[u](1.5,3.5){$4$}
\psline{<->}(2.6,4.6)(3,4.6)
\uput[d](2.8,4.6){$\mu_3$}
\pscircle[linestyle=dashed,linecolor=gray](2.6,4.6){.4}
\psline{<->}(7,3.5)(10,3.5)
\uput[u](8.5,3.5){$4$}
\psline[linecolor=red](0,0)(0,5)
\uput[l](0,2.5){\color{red}{$\Gamma_N$}}

\psline[linecolor=green](10,0)(10,5)
\uput[r](10,2.5){\color{green}{$\Gamma_D$}}

\psline(6.5,.5)(7,-.5)
\uput[dr](7,-.5){$\Omega$}
\end{pspicture}
\end{center}
Our continuous field variable $u_e=u_e(\mu) \in X_e$ satisfies:
\begin{equation}
	\label{e:continuous}
	\left\{ \begin{array}{l}
			 \Delta u_e = 0 \text{ in } \Omega \\
			 u_e = 0 \text{ on } \Gamma_D \\
			 \frac{\partial u_e}{\partial n} = -1 \text{ on } \Gamma_N \\
			 \frac{\partial u_e}{\partial n} = 0 \text{ on } \partial\Omega\setminus(\Gamma_N\cup\Gamma_D) \end{array} \right. 
\end{equation}
	 where
	 \[ X_e=\{ v \in H^1(\Omega) \text{ s.t. } v|_{\Gamma_D}=0\}, \]
	 $\Delta$ denotes the Laplace operator, and $\frac{\partial}{\partial n}$ is the normal derivative with respect to $\partial\Omega$.

	 This continuous variable denotes the potential of a steady, incompressible flow moving in a tube whose profile is given by $\Omega$, with open ends on $\Gamma_N$ and $\Gamma_D$. The Neumann boundary condition on $\Gamma_N$ states that the fluid enters by $\Gamma_N$ with \remove{unit speed} \add{velocity equal to one}, the condition on $\partial\Omega\setminus(\Gamma_N\cup\Gamma_D)$ states that the velocity field is tangential to the boundary of the tube; finally the Dirichlet condition on $\Gamma_D$ guarantees well-posedness, as the potential field is determinated up to a constant.

\remove{The problem \eqref{e:continuous} is equivalent to the following variational formulation} \add{The variational formulation of \eqref{e:continuous} states as follows}: find $u_e=u_e(\mu)\in X_e$ so that:
\[ \int_\Omega \nabla u_e \cdot\nabla v = - \int_{\Gamma_N} v, \;\;\;\forall v\in X_e. \]
This variational problem is well-posed, as the bilinear form $(u,v)\mapsto \int_\Omega \nabla u \cdot \nabla v$ is coercive on $X_e$ (see, for instance,  \cite{toselli2005domain}, lemma A.14).

The above variational problem is discretized using a finite triangulation $\mathcal T$ of $\Omega$ and the associated $P^1(\mathcal T)$ (see \cite{ciarlet2002finite} or \cite{quarteroni2008numerical}) finite element subspace: find $u\in X$ so that
\[ \int_\Omega \nabla u \cdot\nabla v = - \int_{\Gamma_N} v \;\;\;\forall v\in X, \]
where $X=\{ v \in \mathbf P^1(\mathcal T) \text{ s.t. } v|_{\Gamma_D}=0\}$. 

In our experiments, $\dim X=525$.

The affine decomposition of the matrix of the bilinear form in the left-hand side of the above equation is obtained by using a piecewise affine mapping from $\Omega(\mu)$ to a reference domain $\bar\Omega$ as explained in \cite{quarteroni2011certified}, page 11. 

Our scalar output of interest is taken to be:
\[ s(\mu) = \int_{\Gamma_N} u(\mu), \]
and \remove{$\mathcal P$ is endowed with the uniform distribution} \add{$\mu$ has uniform distribution on $\mathcal P$}.

\subsection{Results}
We now present the numerical results obtained using the different error bounds on the output of the model described above. We report our bounds on the non-corrected and corrected outputs, as well as the dual-based output bound. Note that the stability constant $\alpha(\mu)$ is taken as the exact inf; this clearly advantages the dual-based output bound.

For the comparisons to be fair, one should compare the error bounds of same online cost. It is widely assumed that there exists a constant $C$ so that this cost is $C \times 2 (\dim\widetilde X)^3$ for the dual-based method, and $C (\dim\widetilde X)^3$ for our method, since dual-based method involves online inversion of two linear systems of size $\dim\widetilde X$, and one system of the same size for our method. Hence, the reported reduced basis sizes for the dual method are multiplied by a factor $\sqrt[3]{2}$.

In all cases, the reduced bases are computed using POD with snapshot size 80. To compute $\widehat G$, we use a snapshot of size $200$. We also took $K=1$ (ie., a trivial partition of $\mathcal P$). The truncation index $N$ is taken equal to 20. We used a discrete minimization procedure to estimate the $\beta_i$ constants. 


In Figure \ref{f:1}, we compare the different error bounds on the non-corrected, and corrected output. For instance, for the error bound on the non-corrected output, we plot:
\[ \bar\epsilon=\frac{1}{\#S}\sum_{\mu\in S} \widehat\epsilon(\mu,\alpha,N,\Phi) \]
where $S$ is a random subset of $\mathcal P$ with size 200 and $\widehat\epsilon(\mu,\alpha,N,\Phi)$ is defined at \eqref{e:defheps}. Other error bound means are computed accordingly. 


\begin{figure}
\begin{center}
\includegraphics[width=13cm]{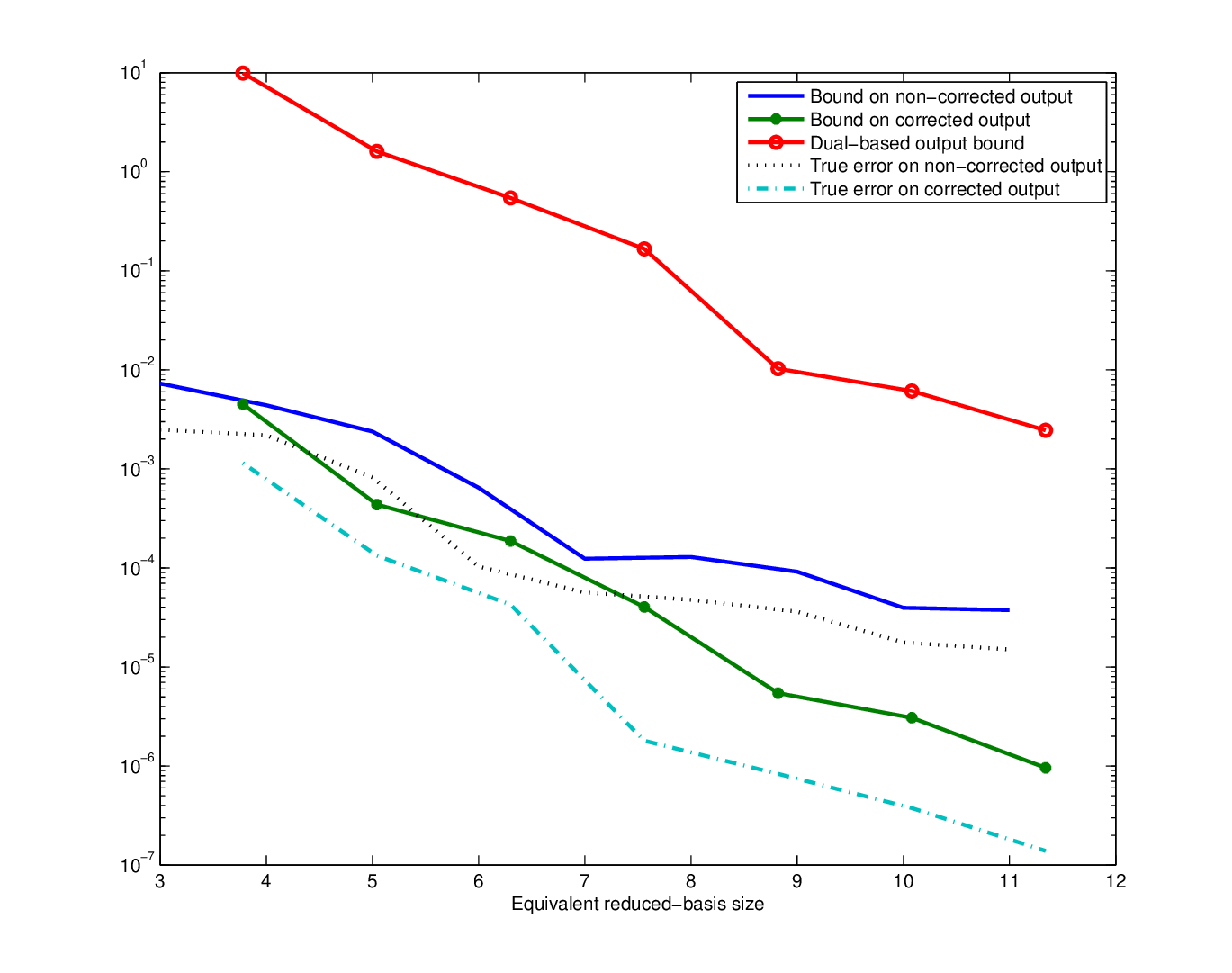}
\caption{{\small Comparison of the mean error bound on the non-corrected output, the mean dual-based error bound ($\epsilon_{cc}$) and the mean error bound on the corrected output (for risk $\alpha=0.0001$). The ``equivalent'' reduced basis sizes are in abscissae. }}
\label{f:1} 
\end{center}
\end{figure}


We also computed the mean of the Lipschitz error bound $\epsilon^L$. It is not reported here as it was way higher than dual-based output error bound. We observe that our new output error bound outperforms  the dual-based error bound, at least for finite reduced bases sizes. Two reasons can be seen to this superiority: the difference in nature (probabilistic vs. sure) between the two bounds, and the fact that we make a crucial use of expansion \eqref{e:decerr}  instead of using a Cauchy-Schwarz (or duality norm) argument. The rate of convergence (slope of the curve) of the corrected output is better than the non-corrected one, and this fact is reported by the two error bounds. Also, the expectation of $\widehat T_2$ was estimated at $10^{-12}$, which allows to choose a  low target risk and remain very competitive, as the $\alpha$-dependency of the bound is in $\widehat T_2/\alpha$.  

\subsection{Application to sensitivity analysis}
We estimate confidence intervals for the sensitivity indices of $s(\mu)$ by using the method described in \cite{janon:inria-00567977}, together with Corollary \ref{corolle}, and the non-corrected output.

We take $M=1000$ as sample size, $B=500$ as number of bootstrap replications (this parameter is used in the procedure which provides the confidence intervals from $B$ replications of $\widehat S_i$, see \cite{janon:inria-00567977}), $\dim\tilde X=10$ as reduced basis size, $\alpha=0.00001$ as output error bound risk, and $\alpha_{as}=0.05$ as Monte-Carlo risk. The level of the combined confidence interval $\left[\widehat S_{i,\alpha_{as}/2}^m;\widehat S_{i,1-\alpha_{as}/2}^M \right]$ is then $(1-\alpha_{as}) (1-\alpha)^M > 0.93$.

The results are gathered in Table \ref{t:2}.  The spread between $\widehat S_i^m$ and $\widehat S_i^M$ accounts for the  \emph{metamodel-induced} error in the estimation of the Sobol indices. The remaining spread between $\widehat S_{i,\alpha_{as}/2}^m$ and $\widehat S_{i,1-\alpha_{as}/2}^M $ is the impact of the sampling error (due to the replacement of the variances in the definition of the Sobol indices by their empirical estimators). We see that, in this case, the metamodel-induced error (certified by the use of our goal-oriented error bound) is very small with regard to the sampling error. We also notice that the estimate for the Sobol index for $\mu_3$ is negative; this is not contradictory as it is the true value of the index that is in $[0,1]$. For small indices, the estimate can be negative.

{\scriptsize
\begin{table}
\begin{center}
\begin{tabular}{|l|l|l|}
	\hline Input parameter & $\left[ \widehat S_i^m; \widehat S_i^M \right]$ & $\left[\widehat S_{i,\alpha_{as}/2}^m;\widehat S_{i,1-\alpha_{as}/2}^M \right]$ \\
	\hline $\mu_1$ & [0.530352;0.530933]   & [0.48132; 0.5791] \\
	\hline $\mu_2$ & [0.451537;0.452099]   & [0.397962;0.51139] \\
	\hline $\mu_3$ & [0.00300247;0.0036825]& [-0.0575764;0.0729923] \\
	\hline
\end{tabular}
\caption{{\small Results of the application of Section \ref{s:appSA} to the estimation of the Sobol indices of the output of our benchmark model. }}
\label{t:2} 
\end{center}
\end{table}
}

\section{Numerical results II: transport equation}
\label{s:numres2}
We now apply our error bound on a non-homogeneous linear transport equation. Compared to the previous example, the considered PDE is of a different kind (hyperbolic rather than elliptic).

\subsection{Benchmark problem}
In this problem, the continuous field $u_e=u_e(x,t)$ is the solution of the linear transport equation:
\[ \frac{\partial u_e}{\partial t}(x,t) + \mu \frac{\partial u_e}{\partial x}(x,t) = \sin(x) \exp(-x) \]
for all $(x,t)\in]0,1[\times]0,1[$, satisfying the initial condition:
\[ u_e(x,t=0)=x(1-x)\;\;\;\forall x\in[0,1], \]
and boundary condition:
\[ u_e(x=0,t)=0\;\;\;\forall t\in[0,1]. \]
The parameter $\mu$ is chosen in $\mathcal P=[0.5, 1]$ and $\mathcal P$ is endowed with the uniform measure.

We now choose a spatial discretization step $\Delta x>0$ and a time discretization step $\Delta t>0$, and we introduce our discrete unknown $u=(u_i^n)_{i=0,\ldots,N_x;n=0,\ldots,N_t}$ where
\[ N_x=\frac{1}{\Delta x}, \;\;\;\text{and}\;\;\; N_t=\frac{1}{\Delta t}. \]
We note here that the considered PDE is hyperbolic and time-dependent, and that we perform the reduction on the space-time unknown $u$, of dimension $(N_x+1)\cdot(N_t+1)$. This is different from reducing the space-discretized equation at each time step.

The $u$ vector satisfies the discretized initial-boundary conditions:
\begin{equation}\label{e:c1} \forall i, \;\; u_i^0 = (i\Delta x)(1-i\Delta x)\end{equation} 
\begin{equation}\label{e:c2} \forall n, \;\; u_0^n = 0 \end{equation}
and the first-order upwind scheme implicit relation:
\begin{equation}\label{e:c3} \forall i, n \;\; \frac{u_{i+1}^{n+1}-u_{i+1}^n}{\Delta t} + \mu \frac{u_{i+1}^{n+1}-u_{i}^{n+1}}{\Delta x}=
\sin(i\Delta x) \exp(-i \Delta x).\end{equation}
\remove{Let's} \add{Let us} denote by $B=B(\mu)$ (resp. $y$) the matrix (resp. the vector) so that \eqref{e:c1},\eqref{e:c2} and \eqref{e:c3} are equivalent to:
\begin{equation}\label{e:c4} B u = y \end{equation}
that is:
\begin{equation}\label{e:c5} B^T B u = B^T y, \end{equation}
so that equation \eqref{e:c5} is \eqref{e:refprob} with $A(\mu)=B^T B$ and $f=B^T y$.

The output of interest is: $s(\mu)=u_{N_x}^{N_t}$. In the following, we take $\Delta t=0.02$ and $\Delta x=0.05$. As in the previous example, the true stability constants are computed for the dual-based error bound.

\subsection{Results}
\subsubsection{Comparison of the bounds}
We took a very low risk level $\alpha=0.0001$, a snapshot size of 200, $N=20$ retained $\widehat\phi_i^G$ vectors and $K=1$. The results (Figure \ref{f:2}) show that, once again, the error bounds we propose in this paper outperforms the dual-based error bound.

\begin{figure}
\begin{center}
\includegraphics[width=13cm]{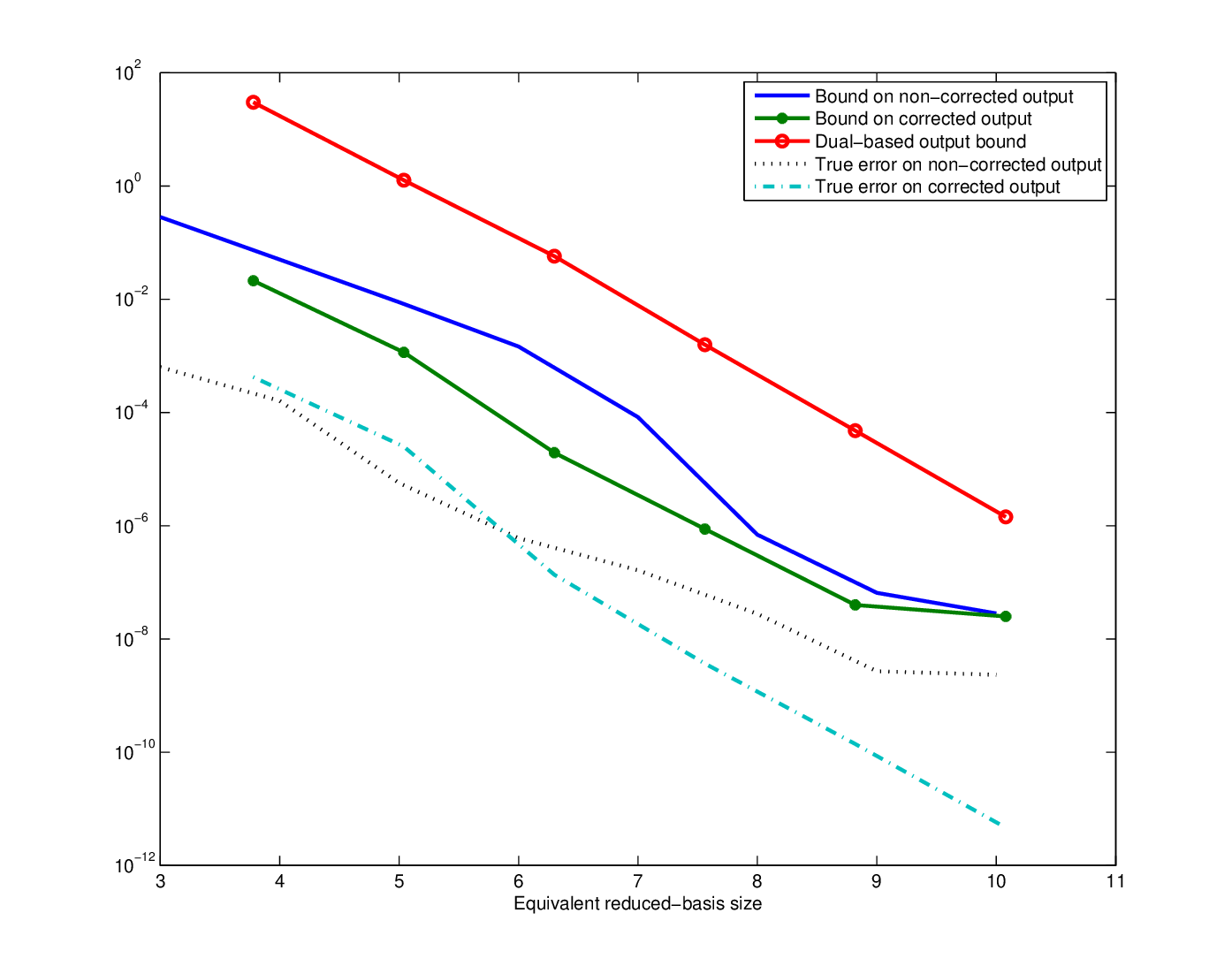}
\caption{{\small Comparison between the mean error bound (corr. and non-corr. outputs), the mean dual-based error bound, and the true (corr. and non-corr.) errors, for different reduced basis sizes.}}
\label{f:2} 
\end{center}
\end{figure}

\subsubsection{Choice of basis}
For the comparison to be fair, we have also checked that using a POD basis does not penalize the dual-based error bound ($\epsilon^{cc}$), by comparing the performance of this error bound when using a POD basis and a so-called ``Greedy'' \cite{buffa2009apriori} procedure. This procedure has a smaller offline cost than the POD, as it requires less resolutions of the reference problem \eqref{e:refprob}. The results, shown in Figure \ref{f:3}, show that the Greedy procedure yields to  inferior performance.  

\begin{figure}
\begin{center}
\includegraphics[width=6cm]{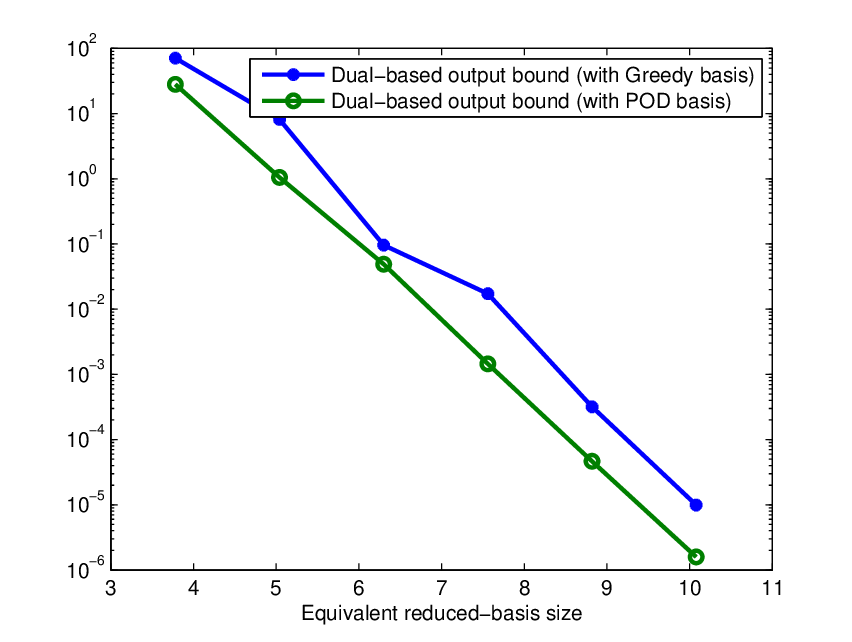}
\caption{Mean of dual-based error bound, when using a POD or a Greedy basis selection procedure, for different reduced basis sizes. }
\label{f:3} 
\end{center}
\end{figure}

\subsubsection{Coverage of the error bound}
 We have also estimated the actual error risk by doing computations of the error bound (on the corrected output), together with the actual error on a random sample of parameters of size $20000$, using $10$-sized bases for primal and dual problem. On this random sample, the estimated error bound is always greater than the true error; hence the error risk estimation appears conservative. 

\section*{Conclusion}
We have presented a new explicitly computable output error bound for the reduced-basis method. We have shown, on two different practical examples, that this bound is clearly better than the naive Lipschitz bound and that, at the expense of a slight, controllable risk, the performances of this new bound are better than the  ones of the existing dual-based output error bound.

\paragraph{Acknowledgements. } This work has been partially supported by the French National Research Agency (ANR) through COSINUS program (project COSTA-BRAVA nr. ANR-09-COSI-015). We thank Anthony Nouy (\'Ecole Centrale de Nantes) and Yvon Maday (Universit\'e Paris 6) for fruitful discussions, and the two anonymous referees for their pertinent remarks, which have greatly improved the quality of the paper.

\appendix
\section{Estimation of $T_2(N,\Phi)$: error analysis}
\label{appendix}
We now assess the error of estimation of $T_2(N,\Phi)$, and its consequence on the risk of the error bound $\widehat\epsilon(\mu,\alpha,N,\Phi)$.

First, by repeating the proof of Theorem \ref{t:1}, and by replacing (for the application of Markov inequality) $T_2(N,\Phi)/\alpha$ by $\widehat T_2(N,\Phi)/\alpha$, one gets
\[ P\left(\abs{s(\mu)-\widetilde s(\mu)} > T_1\Dddmu + \frac{\widehat T_2(N,\Phi)}{\alpha} \right) \leq \alpha \frac{T_2(N,\Phi)}{\widehat T_2(N,\Phi)}. \]
If $T_2(N,\Phi)=0$, we have that 
\[ P\left(\abs{s(\mu)-\widetilde s(\mu)} > T_1\Dddmu + \frac{\widehat T_2(N,\Phi)}{\alpha} \right) = 0, \]
hence the computable error bound has zero risk. So we can assume that $T_2(N,\Phi) \neq 0$.

The error on the risk of $\widehat\epsilon(\mu,\alpha,N,\Phi)$ is not easily attainable, because the same $\Xi$ set of parameters is used to choose $\Phi$ and to compute $\widehat T_2(N,\Phi)$, leading to probabilistic dependence in the family:
\[ \left\{ \left| \sum_{i=N+1}^{\mathcal N} \pscal{w(\mu),\Phi_i}\pscal{r(\mu),\Phi_i} \right| ; \mu\in\Xi\right\}. \]
However, one can take another random sample $\Xi' \subset \mathcal P$ of parameters, independent of $\Xi$, with size $M=\#\Xi'$, and define $\widehat T_2'(N,\Phi)$, the following estimator of $T_2(N,\Phi)$:
\[ \widehat T_2'(N,\Phi) = \frac 1 M \sum_{\mu\in\Xi'} \left| \sum_{i=N+1}^{\mathcal N} \pscal{w(\mu),\Phi_i}\pscal{r(\mu),\Phi_i} \right|, \]
which gives in turn another error bound:
\[ \widehat\epsilon'(\mu,\alpha,N,\Phi)=T_1\Dddmu+\frac{\widehat T_2'(N,\Phi)}{\alpha}, \]
which is also computable in practice, but requires more solutions of the reference problem \ref{e:refprob} during the offline phase.

Notice that $\widehat T_2'(N,\Phi)$ is a random variable with respect to the probability measure used to sample the $\Xi'$ set. We denote by $P'$ this probability measure. We denote by $\rho$ the risk majorant of the computable error bound:
\[ \rho = \alpha \frac{T_2(N,\Phi)}{\widehat T_2'(N,\Phi)}, \]
which is also a random variable with respect to $P'$.

We have the following theorem:
\begin{theorem}
Let \[ B(\mu) = \left| \sum_{i=N+1}^{\mathcal N} \pscal{w(\mu),\Phi_i}\pscal{r(\mu),\Phi_i} \right|, \;\; \sigma^2 = \Var_{P'}(B(\mu)). \]
We have:
\[ \sqrt M \left( \rho - \alpha \right)  \underset{P'}{\longrightarrow} \mathcal N\left(0, \alpha^2 \frac{\sigma^2}{T_2^2} \right) \]
where $ \underset{P'}{\longrightarrow} $ denotes convergence in $P'$-distribution when $M \rightarrow + \infty$ and $\mathcal N(0,\sigma)$ is the centered gaussian distribution with variance $\sigma^2$.
\end{theorem}
Hence $\rho$ converges (in a probabilistic sense) to $\alpha$ with rate $1/\sqrt M$.
\begin{proof}
We have:
\[ T_2(N,\Phi) = \E_{P'}(B(\mu)), \;\; \widehat T_2' = \frac{1}{M} \sum_{\mu\in\Xi'} B(\mu). \]
Hence, by the central limit theorem \cite{van2000asymptotic},
\[ \sqrt M (\widehat T_2 - T_2) \underset{P'}{\longrightarrow} \mathcal N(0, \sigma)  \]
Now we define $g(x) = \alpha T_2(N,\Phi)/x$. As $T_2(N,\Phi) \neq 0$, $g$ is differentiable in $T_2(N,\Phi)$ and one can use the Delta method \cite{van2000asymptotic} to write:
\[ \sqrt M \left( g(\widehat T_2'(N,\Phi)) - g(T_2(N,\Phi)) \right)  \underset{P'}{\longrightarrow} \mathcal N\left(0, \alpha^2 \frac{\sigma^2}{T_2^2} \right), \]
which proves the theorem.
\end{proof}

\bibliographystyle{plain}
\bibliography{biblio}

\end{document}